\documentclass[11pt]{article}
\usepackage{latexsym,amssymb,amsmath,amsthm,amscd}
\usepackage{algorithmic}
\usepackage{amssymb}
\usepackage{times}
\usepackage{enumerate}
\usepackage[T1]{fontenc}
\usepackage[usenames,dvipsnames]{pstricks}
\usepackage{epsfig}
\usepackage{url}
\usepackage{hyperref}

\def\QED{\hfill$\Box$}

\newcommand{\nomargin}[1][0]{\setlength{\belowdisplayskip}{#1pt} \setlength{\belowdisplayshortskip}{#1pt}\setlength{\abovedisplayskip}{#1pt} \setlength{\abovedisplayshortskip}{#1pt}}

\renewenvironment{proof}{\noindent {\bfseries Proof. }}{\QED}

\newtheorem{theorem}{Theorem}[section]
\newtheorem{lemma}       [theorem]{Lemma}
\newtheorem{corollary}   [theorem]{Corollary}
\newtheorem{proposition} [theorem]{Proposition}
\newtheorem{remark}      [theorem]{Remark}

\newtheorem{conjecture}  [theorem]{Conjecture}
\newtheorem{definition}  [theorem]{Definition}

\newtheorem{th-algorithm}[theorem]{Theorem}

\newcommand{\Waring}{\ensuremath \text{Waring}}

\newcommand{\SDE}{\ensuremath \text{SDE}}

\newcommand{\KK}{\mathbb{K}}

\newcommand{\NN}{\mathbb{N}}
\newcommand{\RR}{\mathbb{R}}
\newcommand{\CC}{\mathbb{C}}
\newcommand{\ZZ}{\mathbb{Z}}
\newcommand{\mP}{\mathcal{P}}

\title{On the linear independence of
  shifted powers}
\author{Pascal Koiran, Timoth\'ee Pecatte\\
LIP\thanks{UMR 5668 Ecole Normale Sup\'erieure de  Lyon, CNRS, UCBL, INRIA.
Emails:  
{\tt [Pascal.Koiran, Timothee.Pecatte]@ens-lyon.fr,iggarcia@ull.es} 
}, 
Ecole Normale Sup\'erieure de Lyon, Universit\'e de Lyon\\
Ignacio Garc\'ia-Marco \footnote{Corresponding author. Address: Aix-Marseille Universit\'e. I2M UMR 7373,
Case Courrier 907. 163 avenue de Luminy. 13288 Marseille Cedex 9, France.}\\
I2M, Aix-Marseille Universit\'e
}

\begin{document}
\maketitle

\begin{abstract}
  We call {\em shifted power} a polynomial of the form $(x-a)^e$.
  The main goal of this paper is to obtain broadly applicable criteria
  ensuring that the elements of a finite family $F$ of shifted powers
  are linearly independent or, failing that,
to give a lower bound on the dimension of the space of polynomials spanned
by $F$.
In particular, we give simple criteria ensuring that the dimension
of the span of $F$ is at least $c.|F|$ for some absolute constant $c<1$.
We also propose conjectures implying the linear independence of the elements
of $F$. These conjectures are known to be true for the field of real numbers,
but not for the field of complex numbers. The verification of these conjectures
for complex polynomials directly imply new lower bounds in algebraic complexity.
\end{abstract}



\section{Introduction}

In this article, we %will
consider families of univariate  polynomials of the %PK following
form: $$F = \{ (x-a_i)^{e_i} \;:\: 1 \leq i \leq s \},$$
%where $a_i \in \KK$ and $e_i \in \NN$.
where $e_i \in \NN$ and the $a_i$ belong to a field $\KK$ of characteristic~0.
An element of $F$ will be called a \emph{shifted power}
%PK add:
(polynomials of this form are also called {\em affine powers} in~\cite{GKP}).
%In the following it will always be assumed
We always assume that $(a_i, e_i) \not= (a_j, e_j)$ for $i \not= j$, i.e. that $F$ does not contain the same element twice.

The main goal of this paper is to obtain broadly applicable criteria
ensuring that the elements of $F$ are linearly independent or, failing that,
to give a lower bound on the dimension of the space of polynomials spanned
by $F$.
For instance, we have the following well-known result for the case of equal
exponents.
\begin{proposition}[Folklore]\label{prop:WaringLinearIndependence}
	For any integer $d$, for any distinct $(a_i) \in \KK^{d+1}$, the family $\{(x-a_0)^d, \dots, (x-a_d)^d\}$ is a basis of the space of polynomials of degree at most $d$. %$\KK_{d+1}$.
\end{proposition}

This can be shown for instance by checking that the Wronskian determinant
of the family is not identically 0.
Nullity of the Wronskian is a necessary and sufficient condition for
the linear independence of polynomials~\cite{Bocher,BoDu,Krusemeyer}, so our problem always reduces
in principle to the verification that the Wronskian of $F$ is nonzero.
Unfortunately, the resulting determinant looks hardly manageable
in general. As a result, little seems to be known in the case of unequal exponents (the case of equal exponents is tractable because the Wronskian determinant becomes  a Vandermonde matrix after multiplication of rows by constants).
One exception is the so-called Jordan's lemma  \cite{GraceYoung} 
(see \cite[Lemma 1.35]{IaKa} for a recent reference), which provides the following generalization of Proposition~\ref{prop:WaringLinearIndependence}:

\begin{lemma}\label{JordanLemma}Let $d \in \ZZ^+$, $e_1,\ldots,e_t \in \{1,\ldots,d\}$, and let $a_1,\ldots,a_t \in \KK$ be distinct constants. If $\,\sum_{i = 1}^t (d+1-e_i) \leq d+1$, then  the elements of 
$$\bigcup_{i = 1}^t \ \left\{ (x-a_i)^{e_i},\  (x-a_i)^{e_i+1}, \ldots, (x-a_i)^{d} \right\}$$ are linearly independent.
\end{lemma}

So far, we have only discussed sufficient conditions for linear independence.
The following ``P\'olya condition'' is an obvious
{\em necessary condition}:%PK for linear independence.

\begin{definition} \label{polya}
  For a sequence $e = (e_1, \dots, e_s)$ of integers, let $n_i = |\{j : e_j < i \}|$.
  We say that $e$ satisfies the P\'olya condition if $n_i \leq i$ for all $i$.
  
  For a family $F = \{ (x - a_i)^{e_i} : 1 \leq i \leq s\}$, we say that $F$ satisfies the P\'olya condition if $e = (e_1, \dots, e_s)$ does.
\end{definition}

The name {\em P\'olya condition} is borrowed
from the theory of Birkhoff interpolation~\cite{Lorentz84,GMK15}.
  This necessary condition for linear independence is not  sufficient: for instance we have the linear dependence relation $(x+1)^2-(x-1)^2-4x=0$.
As we shall see later, the P\'olya condition turns out to be sufficient in a probabilistic sense: if the shifts $a_i$ are taken uniformy at random, the resulting family is linearly independent with high probability.
As pointed out above, little is known about deterministic sufficient conditions for linear
independence. But there is an exception when $\KK$ is the field of real numbers:
in this case, some recent progress was made in~\cite{GMK15} thanks to a connection
between Birkhoff interpolation and linear independence of shifted powers.

In particular, the authors
%proves the following sufficient condition in order to have linear independence of a family of real polynomials:
showed that the P\'olya condition is only a factor of 2 away from being also
a  sufficient condition for linear independence:

\begin{theorem}[Theorem 3 in \cite{GMK15}]\label{thm:BirkLinearIndependence}
  %	For a family $F = \{ (x - a_i)^{e_i} : 1 \leq i \leq s\}$, we define $n_i = |\{j : e_j < i \}|$ and $d = \max e_i$.
  Let $F$ and the $n_i$'s be as in Definition~\ref{polya}, and let $d= \max e_i$.
  If all the $a_i$'s are real, and $n_1 \leq 1, n_j + n_{j+1} \leq j+1$ for all $j=1 \dots d$,
  then  the elements of $F$ are linearly independent.
\end{theorem}
They also gave an example of linear dependence that violates only one of the inequalities of Theorem~\ref{thm:BirkLinearIndependence}, showing that this result is essentially optimal.

Theorem~\ref{thm:BirkLinearIndependence} fails badly over the field of complex numbers, as shown by this example from~\cite{GMK15}.
\begin{proposition}\label{dependenceincomplex}
Take $k \in \mathbb{Z}^+$ and let $\xi$ be a $k$-th primitive root
of unity. Then, for all $d \in \mathbb{Z}^+$ and all $\mu \in
\mathbb{C}$ the following equality holds:
\begin{equation} \label{eq:dependence}
  \sum_{j = 1}^{k} \xi^j (x + \xi^j \mu)^d = \sum_{i \equiv -1\, ({\rm mod}\ k) \atop 0 \leq i \leq d} k \binom{d}{i} \mu^i x^{d-i}.
 \end{equation}
\end{proposition}
    Indeed, if we take  $k = \sqrt{d}$ in~(\ref{eq:dependence}),
      Proposition~\ref{dependenceincomplex} shows that  the $2k = 2 \sqrt{d}$ shifted powers in  the set 
  $$\{ (x + \xi^j \mu)^d \, \vert \, 1 \leq j \leq k\} \cup \{ x^{d-i} \, \vert \, i \equiv -1\, ({\rm mod}\ k), 0 \leq i \leq d\}$$
  are linearly dependent;
and the exponents of these shifted powers clearly satisfy the hypothesis of Theorem \ref{thm:BirkLinearIndependence}.
The question of finding a ``good'' sufficient condition for linear independence over $\CC$ was left open in~\cite{GMK15}.
We propose the following conjectures.
\begin{conjecture} \label{conj:bigexp}
  There are absolute constants $a$ and $b$ such that for all large enough integers $s$, the elements in any family $F = \{ (x - a_i)^{e_i} : 1 \leq i \leq s\}$ of $s$ complex shifted powers are linearly independent if $e_i \geq as+b$ for all $i$.
\end{conjecture}
If this conjecture holds true, one must have $a \geq 1$. Indeed, a family
where $e_i \leq s-2$ for all $i$ will violate the P\'olya condition.
One can say more. Indeed, we will see in Proposition~\ref{prop:realLinearIndependence} that the counterpart of Conjecture~\ref{conj:bigexp} for the field of real numbers holds true,
and that one may take $a=2,b=-4$.
We will also show that this result is best possible over $\mathbb{R}$:
one cannot take $a=2,b=-5$.
Therefore, if Conjecture~\ref{conj:bigexp} holds true one must in fact have
$a \geq 2$.

\begin{conjecture} \label{conj:GMK}
  There are absolute constants $a$ and $b$ such that for any
  linearly independent family $F = \{ (x - a_i)^{e_i} : 1 \leq i \leq s\}$
  with $s \leq ad+b$, $a_i \in \mathbb{C}$ and $e_i \leq d$ for all $i$,
  the  elements in the family $F \cup \{(x+1)^{d+1},x^{d+1}\}$ are linearly independent as well.
\end{conjecture}
By Theorem~13 from~\cite{GMK15}, the counterpart of this conjecture for the field of real numbers holds true.
We will not prove any of these two
conjectures in the present paper, but
we will show that a weak version of Conjecture~\ref{conj:bigexp} holds true:
it suffices to replace the hypothesis $e_i \geq as+b$ by the stronger assumption
$e_i \geq s(s-1)/2.$
Moreover, we show that any family of $s$ shifted powers
satisfying the P\'olya condition spans a space of dimension at least
$s/4$ (this should be compared to the straightforward
$\sqrt{s}$ lower bound).

\subsection*{Complexity-theoretic motivation}

Studying the linear independence of shifted powers is a natural and challenging 
problem in its own right, but it also has a complexity-theoretic motivation.
The motivating problem is to prove lower bounds on the size
of the representation of ``explicit'' polynomials as linear combinations
of shifted powers. More precisely, we want a lower bound on the number
$s$ of shifted powers needed to represent a polynomial $f$ of degree $d$ as:
\begin{equation} \label{representation}
  f=\sum_{i=1}^s \alpha_i (x-a_i)^{e_i}.
  \end{equation}
Clearly, one cannot hope for more than a lower bound that is linear in $d$, and this was achieved in~\cite{GMK15} for real polynomials.
For complex polynomials, $\Omega(\sqrt{d})$ lower bounds  (i.e., lower bounds of the form $c\, \sqrt{d}$ for some  absolute constant
  $c > 0$) were obtained for the more
general model of sums of powers of low degree polynomials in~\cite{KKPS}.

The connection between this complexity-theoretic problem and the linear independence of shifted powers arises when the ``hard polynomial'' $f$ itself is defined as a linear combination of shifted powers.
In this case, (\ref{representation}) can be rewritten as a linear dependence relation between shifted powers (those on the right-hand side of (\ref{representation}), and those occuring in the definition of $f$). If we can show that
such a linear dependence is impossible for small enough $s$, we have a lower
bound on $s$. This is exactly the way that the lower bounds in~\cite{GMK15} were obtained.
Likewise, Conjecture~\ref{conj:GMK} is equivalent to the following complexity-theoretic conjecture at the end 
of~\cite{GMK15}: the polynomial $f(x)=(x+1)^{d+1}-x^{d+1}$ requires
$s=\Omega(d)$ terms to be represented under form~(\ref{representation}) with all exponents $e_i \leq d$.
The same reasoning shows that a proof of Conjecture~\ref{conj:bigexp} would yield an $\Omega(d)$ lower bound for representations of degree $d$ polynomials satisfying constraints of the form $e_i=\Omega(d)$ for all $i$.

Finally, we note that lower bounds on the dimension of the space spanned
by a family $F$ of $s$ shifted powers do not seem to imply any
complexity-theoretic lower bound if they fall short of full linear independence
($\dim F = s$). They are hopefully a step in this direction, however.

\subsection*{Organization of the paper}

In Section~\ref{sec:SDE} we introduce one of the main tools of this paper: shifted differential equations (or SDE for short). Those are linear differential equations
with polynomial coefficients. The name $\SDE$ is a reference to Neeraj Kayal's lower bound method of shifted partial derivatives \cite{Saptsurvey}.
We have already used shifted differential equations in the design of efficient
algorithms for shifted powers~\cite{GKP}.
The main novelty in the present paper is a technical lemma (see Proposition~\ref{prop:smallequation} and Corollary~\ref{prop:nodesprop:RootSDEMultiplicity}) which helps to deal with families
where some nodes $a_i$ are repeated, i.e., occur in several shifted powers
of the family.
This lemma is then used for the study of linear independence
(in Section~\ref{sec:independence}) and dimension lower bounds
(in Section~\ref{sec:dim}) over the field of complex numbers.

Section~\ref{sec:independence} 
provides sufficient conditions for linear independence. We first show that a version of Conjecture~\ref{conj:bigexp} with $e_i \geq 2s-4$ holds over the field of real numbers, and that this bound is optimal. For the field of complex numbers, we prove a  version
of this conjecture under the stronger assumption $e_i \geq s(s-1)/2$.
Finally we show that when the shifts $a_i$ are chosen at random, then with
high probability the corresponding family of shifted powers is linearly independent for {\em all} choices of exponents $e_i$ satisfying the P\'olya condition.
As an intermediate step, we give an exact formula for the number of P\'olya sequences where the $s$ exponents $e_i$ are all bounded by the same integer $d$.
%TP improved constants -> previous : s(s-1)

In Section~\ref{sec:dim} we give lower bounds on the dimension
of the space of polynomials spanned by a family of $s$ shifted powers.
We first present a straightforward $\sqrt{s}$ lower bound, and then
give lower bounds that are linear in $s$ for the fields of real
and complex numbers.
We note that the complex lower bound is nonconstructive, i.e., the proof
does not yield an explicit independent subfamily of dimension $\Omega(s)$.

\section{%PK Structure of
  Shifted Differential Equations} \label{sec:SDE}

\subsection{Definitions} \label{sec:def}

Shifted differential equations are a particular type of differential equations with polynomial coefficients.

\begin{definition}\label{definitionSDE}
	A \emph{Shifted Differential Equation} ($\SDE$) of parameters $(t,k,l)$ is 
a differential equation of the form
	\[
		\sum\limits_{i=0}^k P_i(x)f^{(i)}(x) = 0
	\]
	where $f$ is the unknown function and the $P_i$ are polynomials in $\KK[x]$, 
	with $\deg(P_i) \leq i+l$ for all $i \in \{0,\ldots,t-1\},$ $\deg(P_i) \leq l$ for all $i \in \{t,\ldots,k\}$ and $P_k \neq 0$.  
	%PK add:
	We refer to the polynomials $P_i$ as the {\em coefficients} of the SDE.
	\\
	The quantity $k$ is called the \emph{order} of the equation 
and the quantity $l$ is called the \emph{shift}.
	We will usually denote such a differential equation by $\SDE(t,k,l)$. 
\end{definition}

 This class of differential equations includes the well studied Fuchsian differential equations. A Fuchsian equation is a linear differential equation in which the
singular points of the coefficients are regular singular points; in particular, these equations are linear differential equations 
with polynomial coefficients of bounded degree (see, e.g,. \cite{Yoshida}).

We want to highlight that the definition of $\SDE(t,k,l)$ also covers the case when $t > k$; thus the $\SDE$ defined here are a generalization of those used in  \cite{KKPS}.

As an example, the shifted power $(x-a)^e$ satisfies the SDE:
$$(x-a)f'-e.f=0,$$
of order 1 and shift 1 with either $t = 0$ or $t = 1$ (this differential equation can also  be seen as a SDE of order $1$ and shift $0$ with $t = 2$). 
We generalize this observation to several shifted powers in Proposition~\ref{prop:smallequation} below.

Shifted differential equations have already been used in \cite{KKPS} to obtain lower bounds and in \cite{GKP} to obtain efficient algorithms.
In these two papers the authors worked with a slightly less general definition
of SDEs: they considered SDEs of order $k$, shift $l$ and the degrees of the $P_i$'s were all of the form $\deg P_i \leq i + l$. This corresponds now to the choice of 
parameters $(k+1,k,l)$ in Definition \ref{definitionSDE}, i.e., with $t = k+1$. The main advantage of adding the new parameter $t$ is that if we have a $\SDE(t,k,l)$ which is satisfied by $t$ linearly independent polynomials, we can
ensure  that the polynomial $P_k$ has low degree; indeed, its degree is at most $l$ (this should be compared with the value $k+l$ of the previous model of SDEs).
The fact that $P_k$ has low degree will play an essential role in some proofs of this work.

The method of shifted differential equations is quite robust, and this allows us to obtain a  key ingredient similar to that of~\cite{GKP}: %PK for our results:
for a family of $s$ shifted powers, there exists a ``small'' $\SDE$ satisfied by all the terms. More precisely:

\begin{proposition}\label{prop:smallequation}
	Let $F = \{ (x-a_i)^{e_i} \;:\; a_i \in \KK, e_i \in \NN, \;1 \leq i \leq s \}$.
    Then for any choice of parameters $(t,k,l)$ such that
    \[
		s(l+k+1) <  (k+1)(l+1) + \frac{t(t-1)}{2}
    \]
    there exists a $\SDE(t, k',l)$ with $k'\leq k$ satisfied simultaneously by the $f_i(x) = (x-a_i)^{e_i}$ for $i = 1,\dots,s$.
\end{proposition}
\begin{proof}
	The existence of this common $\SDE$ is equivalent to the existence of a common nonzero solution for the following equations $(E_r)_{1\leq r\leq s}$ in the unknowns $\lambda_{i,j}$:
	\begin{equation}
		\sum_{i = 0}^{t-1} \left(\sum_{j = 0}^{i+l} \lambda_{i,j} x^j \right) f_r^{(i)}(x) + \sum_{i = t}^k \left(\sum_{j = 0}^l \lambda_{i,j} x^j \right) f_r^{(i)}(x) =0 \tag{$E_r$}
	\end{equation}
	Therefore, there are $(k+1)(l+1) + \frac{t(t-1)}{2}$ unknowns, so we need to show that the matrix of this linear system has rank smaller than $(k+1)(l+1) + \frac{t(t-1)}{2}$.
	We are going to show that for each fixed value of $r \in \{1,\ldots,s\}$, the subsystem $(E_r)$ has a matrix of rank $\leq l + k + 1$.
	In other words, we have to show that the subspace $V_r$ has dimension less than $l + k + 1$, where $V_r$ is the linear space spanned by the polynomials $x^j f_r^{(i)}(x)$, with 
	$(i,j) \in \NN^2$ such that either $0 \leq i \leq t-1$ and $0 \leq j \leq i+l$; or $t \leq i \leq k$ and $0 \leq j \leq l$.
	But $V_r$ is included in the subspace spanned by the polynomials
	\[
		\{ (x-a_r)^{e_r  + j};\ -k \leq j \leq l,\; e_r+j \geq 0\}.
	\]
	This is due to the fact that the polynomials $x^i$ belongs to the span of the polynomials $\{ (x-a_r)^{\ell}\, : \, 0 \leq \ell \leq i\}$.
	Hence, we have that $\dim V_r \leq l + k + 1$. Since the $(E_r)$ subsystem has a matrix of rank $\leq l + k + 1$, the whole system has rank $\leq s(l+k+1)$. Thus,
	there exists a nonzero solution if $(k+1)(l+1) + \frac{t(t-1)}{2} > s (l+k+1)$. 
\end{proof}

\smallskip

Notice that if we assume that the elements of $F$ are linearly independent, it cannot be part of the solution set of a differential equation of order strictly less than $s$.
Thus the differential equation we obtain from previous Proposition is a $\SDE(t,k',l)$ with $s \leq k' \leq k$.
In particular $\deg(P_{k'}) \leq  l$ if we choose $t \leq s$, therefore we will usually set $t = s$ in the following.

\medskip

In the remainder of Section~\ref{sec:def}, we discuss a few choices 
%PK $(s, k, l)$
of interest of the pair $(k,l)$
for a given set of shifted powers $F$ of size $s$.
Notice first that the expression $s(l+k+1) < (k+1)(l+1) + t(t-1)/2$ in Proposition~\ref{prop:smallequation} is symmetric in $k$ and $l$.

\paragraph{Case $l < s$:} 
	we obtain that there exists a $\SDE(s,k,s-1)$ if $s^2 < s$, which is not possible.
	The fact that in general one cannot find a $\SDE(s,k,l)$ with $l < s$ satisfied by a family of $s$ affine powers will be justified in Remark \ref{rmk:lowerBoundShift}.

\paragraph{Case $k = s$:}                                                                                                                                                                                                                                                                                                                                                                                                                                                                                                                                                                                                                                                                                                                                                                                                                                                                                                                                                                                                                                                                                                                                                                                                                                                                                                                                                                                                                                                                                                                                                                                                                                                                                                                                                                                                                                                                                                                                                                                                                                                                                                                                                           
we obtain $l \geq s(s+1)/2$.
This is similar to the $\SDE$ coming from the factorized Wronskian exhibited in the first version of \cite{GKP}.

\paragraph{Case $l = s$:}
we obtain $k \geq s(s+1)/2$.
This choice of parameters yields a result of linear independence over $\CC$, as we will prove in Proposition \ref{prop:ComplexLinearIndependence}.
%Although not written explicitly, this result was already known as a consequence of \cite[Corollary 3.14]{GKP}.

\paragraph{Case $k = \alpha \cdot s$ with $\alpha > 1$:}
it is enough to have $l \geq \frac{2\alpha-1}{2\alpha - 2}\,s$, which shows the tradeoff between $k$ and $l$.
The ``point in the middle'' where $k = l$ corresponds to  $k = l = \left(1+\frac{\sqrt{2}}{2}\right)s$.

\subsection{Roots of coefficients of a differential equation}
%PK with polynomial coefficients}

We now proceed to %PK add:
one of 
the main tools of this paper.
We show that a differential equation with polynomial coefficients satisfied by every element of a family $F$ of shifted powers must have some structure in the roots of its coefficients.
%PK add:
In this section we will use the convenient notation:
\[
	x^{\underline{i}}=x(x-1).\cdots.(x-i+1)
\]
where $i$ is a positive integer (for $i=0$ we set $x^{\underline{i}}=1$).
When $x$ is a nonnegative integer, we have $x^{\underline{i}}=x! / (x-i)!$ for
$x \geq i$ and $x^{\underline{i}}=0$ for $x<i$.
%TP add:
This notation allows us to write the $i$-th derivative of a shifted power $f(x) = (x - a)^e$ in a concise  way: $f^{(i)} = e^{\underline{i}} (x-a)^{e-i}$.

We first make the following remark, which was already present in \cite{GKP}.
\begin{proposition}\label{prop:nodesRootSDE}
	Consider the following differential equation with polynomial coefficients: 
	\[
		\sum_{i = 0}^k P_i(x) f^{(i)}(x) = 0.
	\]
	Assume that $(x - a)^{e}$ satisfies this equation, with $e \geq k$.
	Then we have $P_k(a) = 0$.
\end{proposition}
\begin{proof}
	Since $(x-a)^{e}$ is a solution, we have
 	\[
 		\sum_{i = 0}^k P_i(x) e^{\underline{i}} (x-a)^{e - i} = 0.
 	\]
 	We deduce that there exists $q \in \KK[x]$ such that $P_k(x) (x-a)^{e-k} = (x-a)^{e-k+1} q(x)$, from which we deduce that $P_k(a) = 0$.
\end{proof}

\bigskip

As a direct consequence, if $(x - a)^e$ and $(x - b)^f$ with $a \not= b$ and $e,f \geq k$ both satisfy the same equation, then both $a$ and $b$ are roots of $P_k$.
However, if $(x - a)^e$ and $(x - a)^{f}$ both satisfy the same equation, can we say more than just $P_k(a) = 0$?
We will answer positively this question thanks to the following proposition:

\begin{proposition}\label{prop:RootSDEMultiplicity}
	Let $(*)$ be the following differential equation with polynomial coefficients:
	\begin{equation}
		\sum_{i = 0}^k P_i(x) f^{(i)}(x) = 0 \tag{$*$}
	\end{equation}
	If $(*)$ is satisfied simultaneously by $x^{e_1}, \dots, x^{e_n}$ where
        %PK add:
        $n \leq k$ and
        $e_1 > e_2 > \dots > e_n \geq k - n + 1$, then for all $m = 0 \dots n-1$,
        %PK we have $x^{n - m} | P_{k - m}$.
        $x^{n - m}$ divides $P_{k - m}$.
\end{proposition}
\begin{proof}
	By injecting $x^{e_j}$ into $(*)$, we get the following equation:
	\[
		\sum_{i= 0}^{\min(e_j,k)} P_i(x)\; e_j^{\underline{i}} \; x^{e_j - i} = 0
	\]
	If $e_j \leq k$, we multiply the equation by $x^{k - e_j}$, and otherwise, we factor out $x^{e_j - k}$.
	In both cases, we %PK finally
        obtain:
	\begin{equation}
		\sum_{i=0}^k e_j^{\underline{i}} \;\cdot P_i(x) x^{k - i} = 0  \tag{$E_j$}
	\end{equation}
	We have $n$ such equations $(E_j)_{1 \leq j \leq n}$ and we will now take a "good" linear combination to deduce the result.
	
	Fix an integer $0 \leq m < n$, and consider the $n$-tuple $\vec{u}_m \in \mathbb{K}^n$ such that $(\vec{u}_m)_i = 0$ for all
        %PK $i \not= m$ and $(\vec{u}_m)_m = 1$.
        $i \not= n-m$ and $(\vec{u}_m)_{n-m} = 1$.
	A "good" linear combination is a $n$-tuple $(\alpha_1, \dots, \alpha_n)$ such that there exists $(b_0, \dots, b_{k-n})$ satisfying
	\[
		\begin{pmatrix}
			e_1^{\underline{0}} & \dots & e_n^{\underline{0}} \\ 
			e_1^{\underline{1}} & \dots & e_n^{\underline{1}} \\ 
			\vdots & \ddots & \vdots \\ 
			e_1^{\underline{k-1}} & \dots & e_n^{\underline{k-1}} \\ 
			e_1^{\underline{k}} & \dots & e_n^{\underline{k}}
		\end{pmatrix}
		\cdot
		\begin{pmatrix}
			\alpha_1 \\ 
			\vdots \\ 
			\alpha_n
		\end{pmatrix} 
		=
		\begin{pmatrix}
		b_0 \\ 
		\vdots \\ 
		b_{k-n} \\ 
		\vec{u}_m
		\end{pmatrix} 
		\begin{array}{c}
     		\left. \rule{0pt}{9.5mm} \right\} k-n+1\;\text{ rows}\\
     		\left. \rule{0pt}{3.5mm} \right\} \;n\quad\quad\quad\quad\text{ rows}\\
		\end{array}
	\]
	We claim that it is always possible to find such a tuple.
	Assuming this fact, we then compute the following equation:
	\[
		0 	= \sum_{j=1}^n \alpha_j (E_j)
			= \sum_{i=0}^{k-n} b_i \; P_i(x)x^{k-i} + P_{k - m}(x)x^{m}
	\]
	This directly implies that $x^{n - m}$ divides $P_{k - m}(x)$.

	To prove the claim, we will use the proof technique of Lemma 2 from \cite{LS} to show that the following square submatrix is invertible:
	\[
		\begin{pmatrix}
			e_1^{\underline{k-n+1}} & \dots & e_n^{\underline{k-n+1}} \\ 
			\vdots & \ddots & \vdots \\ 
			e_1^{\underline{k}} & \dots & e_n^{\underline{k}}
		\end{pmatrix}.
	\]
	We first factorize its determinant using the fact that $a^{\underline{b + c}} = a^{\underline{b}} \cdot (a-b)^{\underline{c}}$ to obtain
	\[
		\left|
		\begin{matrix}
			e_1^{\underline{k-n+1}} & \dots & e_n^{\underline{k-n+1}} \\ 
			\vdots & \ddots & \vdots \\ 
			e_1^{\underline{k}} & \dots & e_n^{\underline{k}}
		\end{matrix}
		\right|
		=
		\prod_{i=1}^n e_i^{\underline{k-n+1}}
		\; \cdot \;
		\left|
		\begin{matrix}
			d_1^{\underline{0}} & \dots & d_n^{\underline{0}}\\ 
			\vdots & \ddots & \vdots \\ 
			d_1^{\underline{n-1}} & \dots & d_n^{\underline{n-1}}
		\end{matrix}
		\right|
	\]
	where $d_i = e_i - k+n-1$.
	Notice that the constant we have factorized is non-zero, since $e_i \geq k-n+1$.
	Assume for contradiction that the rows are linearly dependent.
	This implies that there exists a
        %PK add:
        nonzero
        tuple $(\alpha_i)$ such that for all $j = 1 \dots n$, we have $\sum_{i = 0}^{n-1} \alpha_i d_j^{\underline{i}} = 0$.
	In other words, if we consider the polynomial $P(x) = \sum_{i=0}^{n-1} \alpha_i x^{\underline{i}}$, then $P(d_j) = 0$ for all $j$.
	However, all the $d_j$'s are distinct and $P$ is of degree
        %PK add:
        at most
        $n-1$, a contradiction.
\end{proof}

\bigskip

%PK remarque déplacée après le corollaire:
%Building up on this proposition, we can now prove that every nodes in the family has to be a root of the last coefficient, along with its multiplicity.
As a consequence, we obtain the following refinement of Proposition~\ref{prop:nodesRootSDE}. 
\begin{corollary}\label{prop:nodesprop:RootSDEMultiplicity}
	Let $(*)$ be the following differential equation with polynomial coefficients of order $k > 0$:
	\begin{equation}
		\sum_{i = 0}^k P_i(x) f^{(i)}(x) = 0 \tag{$*$}
	\end{equation}
	Consider a family $F = \{ (x-a_i)^{e_i} : 1\leq i \leq s, e_i \geq k \}$ such that $(*)$ is satisfied simultaneously by all the elements of $F$.
	Then $\prod_{i = 1}^s (x-a_i)$ divides $P_k$.
\end{corollary}
In other words, each node $a_i$ is a root of $P_k$ of multiplicity greater or equal to the number of occurrences of this node in the family $F$.

\begin{proof}
	We partition $F$ in subfamilies along the values of the $a_i$'s: $F = \uplus_{i =  1}^t F_i$ where $F_i = \{ (x - b_i)^{\varepsilon_{i,j}} : 1 \leq j \leq s_i\}$, such that $b_i \not= b_j$ for $i \not= j$.
	Notice that $\prod_{i=1}^s (x-a_i) = \prod_{i=1}^t (x-b_i)^{s_i}$, and since $b_i \not = b_j$, it is enough to show that for $i = 1 \dots t$, we have $(x-b_i)^{s_i}$ divides $P_k$.
	This is obtained directly by using Proposition \ref{prop:RootSDEMultiplicity} with $m = 0$.
\end{proof}

\medskip

\begin{remark}\label{rmk:lowerBoundShift} 
	This explains why in general one cannot find a $\SDE$ satisfied simultaneously by a family of $s$ shifted powers with shift $l < s$.
	Indeed, for any choice of parameters $(s,k,l)$ with $k \geq s$, consider a family of shifted powers $F = \{ (x-a_i)^k : 1 \leq i \leq s\}$. 
%	Then $F$ is linearly independent, and therefore for any $\SDE(s,k,l)$ satisfied by $F$,
%	$\sum_{i=0}^k P_i(x)f^{(i)}(x) = 0$,
%	we have that $k \geq s$.
	Take any $\SDE(s,k',l)$ with $k' \leq k$ satisfied by $F$. Since  the elements of $F$ are linearly independent, then $k' \geq s$. Hence, Corollary \ref{prop:nodesprop:RootSDEMultiplicity} yields that the polynomial $\prod_{i=1}^s (x-a_i)$ divides $P_{k'}$, which has degree at most $l$, implying $l \geq s$.
%	This gives a lower bound on the shift of a $\SDE$ satisfied simultaneously by a family of shifted powers.
%	Consider any family of shifted powers of the same degree $F = \{ (x-a_i)^d : 1 \leq i \leq s\}$ and $s \leq d$. 
%	Then, $F$ is linearly independent.
%	Therefore, for any $\SDE(s,k,l)$ satisfied by $F$
%	\[
%		\sum_{i=0}^k P_i(x)f^{(i)}(x) = 0
%	\]
%	we have that $k \geq s$.
%	Moreover, define $p = \max\{i \leq d : P_i \not=0\}$, we claim that $p \geq s$.
%	If it was not the case, then the differential equation $\sum_{i=0}^p P_i(x)f^{(i)}(x) = 0$ would still be satisfied by $F$, but would be of order $< s$, a contradiction.
%	Hence, Corollary \ref{coro:everyNodeMatters} yields that the polynomial $\prod_{i=1}^s (x-a_i)$ divides $P_p$, which has degree at most $l$.
%	This implies that $l \geq s$, meaning that any $\SDE$ satisfied by this family must be of shift at least $s$.
\end{remark}

\medskip

We now remove the hypothesis of ``big exponents'' ($e_i \geq k$) and prove that every %PK nodes
node in the family should appear as a root of one of the coefficients of the differential equation.

\begin{corollary}\label{coro:everyNodeMatters}
	Let $(*)$ be the following differential equation with polynomial coefficients:
	\begin{equation}
		\sum_{i = 0}^k P_i(x) f^{(i)}(x) = 0. \tag{$*$}
	\end{equation}
	with $P_0 \not= 0$.
	Define $I = \{i : P_i(x) \not = 0 \}$.
	Consider a family $F = \{ (x-a_i)^{e_i} : 1\leq i \leq s \}$ such that $(*)$ is satisfied simultaneously by all the elements of $F$.
	Then $\prod_{i = 1}^s (x-a_i)$ divides $\prod_{i \in I} P_i$. \\ 
	Indeed, for a given index $i$, $(x - a_i)$ will divide $P_j$ with $j = \max\{p : p\in I, p \leq e_i \}$.
\end{corollary}
\begin{proof}
	Without loss of generality, we can assume that $e_1 \geq e_2 \geq \dots \geq e_s$, and we write $F = \{f_1, \dots f_s \}$ where $f_i = (x - a_i)^{e_i}$.
	We consider the last index $p$ such that $e_p \geq k$ and partition $F$ into two sets: $F = \{f_1, \dots, f_p\} \cup \{f_{p+1}, \dots, f_{s}\}$.
	Using Corollary \ref{prop:nodesprop:RootSDEMultiplicity}, we get that $\prod_{i = 1}^p (x-a_i)$ divides $P_k$.
	We now consider the following equation:
	\begin{equation}
		\sum_{i = 0}^{k-1} P_i(x) f^{(i)}(x) = 0 \tag{$*'$}
	\end{equation}
	Notice that for $i > p$, $f_i$ satisfies $(*')$.
	Since the order of the equation has decreased, and since $0 \in I$, we can proceed by induction to obtain that $\prod_{i=p+1}^s (x-a_i)$ divides $\prod_{i \in I\setminus \{k\}} P_i$.
	The combination of these two facts yields the wanted result.
\end{proof}

\section{Linear independence} \label{sec:independence}

The goal of this section is to prove sufficient conditions for
the linear independence of a family of
%PK polynomials of $\KK[x]$.
shifted powers.
One trivial such condition is when
%PK add:
the degrees of the polynomials in the family are all distinct.
%PK add:
As pointed out in Proposition~\ref{prop:WaringLinearIndependence},
linear independence also holds in the case of equal exponents.  
%PK déplacé dans l'intro:
%Other well-known independent linear families of $\KK[x]$ are the following:
%Here is another well-known sufficient condition for linear independence:
%\begin{proposition}[Folklore]\label{prop:WaringLinearIndependence}
%	For any integer $d$, for any distinct $(a_i) \in \KK^{d+1}$, the family %$\{(x-a_0)^d, \dots, (x-a_d)^d\}$ is a basis of PK $\KK_{d+1}$.
%\end{proposition}
%\begin{proof}
%	The proof relies on Vandermonde's matrices.
%\end{proof}
%PK Yet, if we relax the hypothesis that the polynomials have same degree in previous Proposition, almost nothing is known.
%This is one of the motivation to study families of shifted powers, and in this section we give some non-trivial sufficient condition for them to be linearly independent, and we discuss about optimality of these results.
In this section we provide further sufficient conditions for linear independence and study their optimality.

\subsection{The real case} \label{sec:realind}

%PK déplacé dans l'intro :

%In \cite{GK}, the authors proves the following sufficient condition in order to have linear independence of a family of real polynomials:
%\begin{theorem}[Theorem 3 in \cite{GK}]\label{thm:BirkLinearIndependence}
%	For a family $F = \{ (x - a_i)^{e_i} : 1 \leq i \leq s\}$, we define $n_i = |\{j : e_j < i \}|$ and $d = \max e_i$.
%	If all the $a_i$'s are real, and $n_1 \leq 1, n_j + n_{j+1} \leq j+1$ for all $j=1 \dots d$, then $F$ is linearly independent.
%\end{theorem}

%They also provides a example of linear dependence that violates only one of the inequality above, showing that this result is essentially optimal.
%\\

In the following we derive
%PK add:
from Theorem~\ref{thm:BirkLinearIndependence}
another sufficient condition for the linear independence of families without small exponents, and show that the result is tight.

\begin{proposition}\label{prop:realLinearIndependence}
	For any family $F = \{ (x - a_i)^{e_i} : a_i \in \RR, e_i \geq \max(1, 2s - 4), 1 \leq i \leq s\}$, the elements of $F$ are linearly independent.
\end{proposition}
\begin{proof}
	Assume without loss of generality that $d = e_1 \geq e_2 \geq \dots \geq e_s$.
	We first eliminate a few trivial cases:
	\begin{itemize}
	  \setlength\itemsep{-0.5mm}
		\item $s = 1$ :  the elements of $F$ are linearly independent.

		\item $s = 2$ : two affine powers are linearly dependent if and only if they are equal, thus the elements of $F$ are linearly independent.

		\item $e_1 = e_2 = \dots = e_s$. Then all the $a_i$ are distinct because all $(e_i,a_i)$ are distinct.
		  Since $d \geq 2s - 4$ and $s \geq 3$, we have $s \leq d +1$ and %PK hence using Proposition \ref{prop:WaringLinearIndependence} we have that
                   the elements of $F$ are linearly independent
                  by Proposition \ref{prop:WaringLinearIndependence}.

		\item $e_1 > e_2$ : no linear dependence could involve $(x-a_1)^{e_1}$, hence it is enough to show that the elements of the subfamily $F' = \{ (x-a_i)^{e_i} : 2 \leq i \leq s \}$ are linearly independent.
			Since $F'$ satisfies the hypotheses of the Proposition, we can deal with this case by induction on $s$.
	\end{itemize}

	We now have $e_1 = e_2 > e_s$ and $s \geq 3$ which implies that $d \geq 2s - 3$ and that $n_d \leq s - 2$.
	Such a family satisfies the hypotheses of Theorem \ref{thm:BirkLinearIndependence}, which directly yields the result.
	Indeed:
	\begin{itemize}
	  \setlength\itemsep{-0.5mm}
		\item For $i = 0$, we have $n_1 = 0 \leq 1$.
		\item For $i \leq 2s - 5$, we have $n_i + n_{i+1} = 0$.
		\item For $2s - 5 < i < d$, we have $n_i + n_{i+1} \leq 2(s - 2) \leq i + 1$.
		\item For $i = d$, we have $n_d + n_{d+1}  \leq  2s - 2 \leq d + 1$.
	\end{itemize}
\end{proof}

\smallskip

In order to show that the bound $e_i \geq 2s - 4$
%PK add:
in the above result
is tight, we will consider the %PK following
real polynomial $H_{2d+1}(x) = (x+1)^{2d+2} - x^{2d+2}$ and write it
%PK in another form, yielding to a linear dependence.
as a sum of $(2d+1)$-st powers.
Define the Waring rank $\Waring_\KK(f)$ of a 
%PK add:
univariate polynomial $f$ as 
%{\nomargin
  %PK $$ \Waring_\KK(f) =
  \begin{equation} \label{waringdef}
    \min \big\{s : \exists \; \alpha \in \KK^s, a \in \KK^s,  \; f = \sum_{i=1}^s \alpha_i (x + a_i)^{\deg f} \big\}
    \end{equation}
    %$$
%}
Waring rank has been studied by algebraists and geometers since the 19th century.
%PK add:
It is classically defined for a homogeneous polynomial of degree $d$ as the smallest number $s$ such that $f$ can be written as a sum of $s$ $d$-th powers
of linear forms.
The Waring rank of a univariate polynomial $f \in \KK[X]$ as defined in~(\ref{waringdef}) is closely related with the Waring rank of $f_H \in \KK[x,y]$, its homogenization with respect to the variable $y$. 
Indeed, it is straightforward to check that: 
\begin{itemize}
	\item $\Waring_\KK (f_H) \leq \Waring_\KK(f)$, and 
	\item if $f_H = \sum_{i = 1}^w (a_i x + b_i y)^d$ with $a_i,b_i \in \KK$  and $a_i \neq 0$ for all $i$, 
		then  $\Waring_\KK(f) \leq w$. 
\end{itemize}

The algorithmic question of computing the Waring rank
%PK add:
of bivariate complex homogeneous polynomials
was solved by Sylvester in 1852.
We refer to \cite{IaKa} for the historical background and in the following proof we will use the algorithm described more recently in \cite{CS}.
We now use this algorithm to show that %PK $H$
$H_{2d+1}$ has a ``large'' Waring rank:
\begin{proposition}
	We have $\Waring_\RR(H_{2d+1}) = \Waring_\CC(H_{2d+1}) = d + 1$.
\end{proposition}
\begin{proof}
	We will use the algorithmic result in \cite[Section 3]{CS} to compute the complex Waring rank of $H_{2d+1}$ and then prove that it coincides with its real Waring rank.
	We consider $P(x,y)$ the homogenization of $H_{2d+1}$ with respect to the variable $y$: $P(x,y) = \sum_{i=0}^{2d+1} \binom{2d+2}{i}  x^{i} y^{2d+1-i} = \sum_{i=0}^{2d+1} \binom{2d+2}{i+1}  x^{2d+1-i} y^{i} $.
	We extract the coefficients $Z_i = \frac{\text{coeff}(P,  x^{2d+1-i} y^{i})}{\binom{2d+1}{i}} = \frac{\binom{2d+2}{i+1}}{\binom{2d+1}{i}} = \frac{2d+2}{i+1}$ and
        %PK add:
        following~\cite{CS},
        we construct the matrix
        %PK following matrix of the $d^{th}$-partial derivatives
	\[
		M = \begin{pmatrix}
			Z_0 & Z_1 & \cdots & Z_d \cr
			Z_1 & Z_2 & \cdots & Z_{d+1} \cr
			\vdots & \vdots & \ddots & \vdots \cr
			Z_{d+1} & Z_{d+2} & \cdots & Z_{2d+1}
		\end{pmatrix}
		=
		(2d+2) \cdot
		\begin{pmatrix}
			\frac{1}{1} & \frac{1}{2} & \cdots & \frac{1}{d+1} \cr
			\frac{1}{2} & \frac{1}{3} & \cdots & \frac{1}{d+2} \cr
			\vdots & \vdots & \ddots & \vdots \cr
			\frac{1}{d+2} & \frac{1}{d+3} & \cdots & \frac{1}{2d+2}
		\end{pmatrix}
	\]
	The last matrix is a Hilbert matrix with an additional row.
	Hilbert matrices are known to be invertible, as special cases of Cauchy matrices, 
	therefore we have $rank(M) = d + 1$, which implies,
        according to \cite{CS} that the complex Waring rank is either $d + 1$ or $d + 2$.
	In order to show that it is in fact $d+1$, we have to find a vector $f \in \CC^{d+2}$ in the kernel of $M^{t}$ (which is unique up to scalar multiplication) and prove that the corresponding polynomial $F(x) = \sum_{i=0}^{d+1} f_i x^i$ does not have multiple roots.

	\noindent
	Notice that the $i^{th}$ row of $M^{t}f$ can be rewritten as
	\[
		(M^{t}f)_i = \sum_{j = 0}^{d+1} \frac{1}{i+j+1} f_j = \sum_{j = 0}^{d+1} \int_0^1 x^i \cdot f_j \, x^j dx= \int_0^1 x^i F(x) dx
	\]
	The equality $M^t f = 0$ can thus be restated as $\langle F, x^i \rangle = 0$ for $i = 0 \dots d$, with the corresponding scalar product $\langle f,g \rangle = \int_0^1 f(x)g(x)dx$.
	Such a polynomial can be obtained by the Gram-Schmidt process to $\{1, x, \dots x^{d+1} \}$ and is classically known as the shifted Legendre polynomial: it can be obtained from the Legendre polynomial by the affine transformation $x \mapsto 2x - 1$.
	A classical result (see, e.g., \cite{Legendre}) is that the Legendre polynomial of degree $d+1$ has $d+1$ distinct real roots in the interval $(-1,1)$.
	Therefore our polynomial $F$ has $d+1$ distinct real roots in the interval $(0,1)$.
	This shows that $\Waring_\CC(H) = d+1$.
	
	Moreover, if we denote by $(a_i)$ the roots of $F$, there %PK exists
        exist coefficients $(\alpha_i)$ such that
        %PK $f = \sum_{i=1}^{d+1} \alpha_i (x - a_i)^{2d+1}$.
        $P(x,y)=\sum_{i=1}^{d+1} \alpha_i (x - a_iy)^{2d+1}$.
	Since the $a_i$ are real, if we take the real part of this equality, we obtain %PK $f = \sum_{i=1}^{d+1} \Re(\alpha_i) (x + a_i)^{2d+1}$,
        $P(x,y)=\sum_{i=1}^{d+1} \Re(\alpha_i) (x - a_iy)^{2d+1}$
        proving that %PK $\Waring_\RR(H) = \Waring_\CC(H) = d + 1$.
        $\Waring_\RR(P) = \Waring_\CC(P) = d + 1$. 
        Since $a_i \neq 0$ for all $i$, we conclude that $\Waring_\RR(H_{2d+1}) = \Waring_\CC(H_{2d+1}) = d + 1$.    
\end{proof}
%PK add:
\begin{remark}
  Up to multiplication of rows and columns by constants, the matrix~$M$ in the above proof is nothing but the matrix of $d$-th order partial derivatives of $P$.
  This explains why the Waring rank of $P$ is greater or equal to the rank of $M$. Again, we refer to~\cite{CS} for a proof that the Waring rank is in fact equal to $rank(M)$ or $d+2 -rank(M)$.
  \end{remark}
\begin{remark}
	A similar proof shows that $\Waring_\RR(H_{2d}) = d + 1$, proving that in general $\Waring_\RR(H_d) = \lceil \frac{d+1}{2} \rceil$.
\end{remark}

\bigskip
\noindent
By this result, there exist $\alpha_1,\ldots,\alpha_{d+1} \in \RR$ and $a_1,\ldots,a_{d+1} \in (0,1)$ such that
\[
	(x+1)^{2d+2} - x^{2d+2} = \sum_{i=1}^{d+1} \alpha_i %PK (x + a_i)^{2d+1}
(x - a_i)^{2d+1}\]
This equality is a linear dependence of $d + 3$ terms of degree at least $2d + 1 = 2(d+3) - 5$, showing the optimality of the bound $e_i \geq 2s - 4$.

\subsection{The complex case} \label{sec:complexind}
In the complex case, we can prove a similar sufficient condition for linear independence:
\begin{proposition}\label{prop:ComplexLinearIndependence}
		For any family $F = \{ (x - a_i)^{e_i} : a_i \in \CC, e_i \geq (s-1)s/2, 1 \leq i \leq s\}$, the elements of $F$ are linearly independent.
\end{proposition}
\begin{proof}
	Take $G \subseteq F$ a  minimal generating subfamily of $F$ and suppose for contradiction that $|G| = t < s$.
	Using Proposition \ref{prop:smallequation}, there exists a $\SDE(t,k,t)$ for some $k \leq (t+1)t/2$ satisfied simultaneously by every element of $G$:
	\[
		\sum_{i = 0}^k P_i(x) f^{(i)}(x) = 0
	\]
	Moreover, since $\langle G \rangle = \langle F \rangle$, we have that every elements of $F$ satisfies this $\SDE$.
	Using Corollary \ref{prop:nodesprop:RootSDEMultiplicity}, since $e_i \geq (s-1)s/2 \geq t(t+1)/2 \geq k$, we thus have that $\prod_{i=1}^s (x-a_i)$ divides $P_k$.
	This yields $s \leq \deg P_k \leq t < s - 1$, a contradiction.
\end{proof}

We do not know if the bound $e_i \geq (s-1)s/2$ is tight.
The best example we know is the one provided in the real case that achieves a linear dependence of $s$ affine powers with $e_i \geq 2s-5$.

\subsection{Genericity and linear independence}
Let $\mP_s$ denote the set of sequences $e = (e_1,\ldots,e_s) \in \NN^s$ satisfying the P\'olya condition.
The goal of this section is to study two different random processes for generating a family of affine powers, and to bound the probability that the elements of the generated family are linearly independent.
In Corollary \ref{coro:highproba}, we study the case where the sequence $e \in \mP_s$ is fixed, and the $a_i$'s are taken uniformly and independently from a set $S$.
In Theorem \ref{th:highproba}, we study the case where we take the $a_i$'s uniformly and independently from a set $S$ and we want them to give independent families for any $e \in \mP_s$.
Notice that Corollary \ref{coro:highproba} does not directly implies Theorem \ref{th:highproba} as the number of sequences in $\mP_s$ is infinite.

We first prove that given a family $F$ of $s$ shifted powers such that $1 \not\in \langle F \rangle$, the number of shifted powers in $\langle F \rangle$ of degree $d$
is upper bounded by a linear expression in $s$.
The condition $1 \not\in \langle F \rangle$ is satisfied if the elements of the set of derivatives $F' = \{ f' : f \in F\}$ are linearly independent.
\begin{proposition}\label{prop:finiteShiftedPowers}
	Consider a family $F = \{ (x - a_i)^{e_i} : 1 \leq i \leq s \}$ such that $1 \notin \langle F \rangle$.
	Let $S(d) = \{ (x - a)^d \in \langle F \rangle\}$.
	\\
	Then for any $e \in \NN$, we have $|S(e)| \leq 2s - 1$.
\end{proposition}
\begin{proof}
	Notice first that for any $e \in \NN$, we must have $|S(e)| \leq e$.
	Otherwise, it would contain a basis of $\KK_e[x]$ and we could obtain $1$ as a linear combination of elements of $F$, which contradicts the hypothesis.
	Therefore we are done if $e \leq 2s - 1$.

	Otherwise, Proposition \ref{prop:smallequation} ensures that there exists a $\SDE(0,k,l)$ with $k = l = 2s - 1$ of order $k' \leq k$ satisfied by all the elements of $F$:
	{\nomargin[6]
	\begin{equation}
		\sum_{i = 0}^{k'} P_i(x) f^{(i)}(x) = 0 \tag{$*$}
	\end{equation}
	}
	Since this equation is satisfied by all the elements of $S(e)$, we use Proposition \ref{prop:nodesprop:RootSDEMultiplicity} to obtain: $|S(e)| \leq \deg P_{k'} \leq l \leq 2s - 1$.
\end{proof}

\medskip

Given a sequence $e \in \mP_s$, we can take the $a_i$'s uniformly and independently in an iterative way and use the previous result to lower bound the probability that  the elements of the resulting family of shifted powers are linearly independent at each step.
The resulting bound is given in Corollary \ref{coro:highproba}.
Its proof requires the following technical lemma.

\begin{lemma}\label{lemmaprojection}
	Let $e = (e_1,\ldots,e_{s+1}) \in \mP_{s+1}$, then $f = (e_1-1,\ldots,e_s-1) \in \mP_s$.
\end{lemma} 
\begin{proof} 
	We set $n_i := |\{j \, : \, e_j < i,\, 1 \leq j \leq s+1\}|$ and $n_i' := |\{j \, : \, e_j - 1 < i,\, 1 \leq j \leq s\}|$.
	Since $e$ satisfies the P\'olya condition we have that $n_i \leq i$ for all $i$. 
	Moreover, since $e$ is non-increasing, we have that 
	\begin{itemize}
		\item $n_i' = 0$ if $i \leq e_{s+1}$,
		\item $n_i' = n_{i+1} - 1 \leq i$ if $i > e_{s+1}$;
	\end{itemize}
	hence, $f$ also satisfies the P\'olya condition.
\end{proof}

\begin{corollary}\label{coro:highproba}
	Let $e = (e_1, \dots, e_s) \in \mP_s$ and let $S$ be a finite subset of $\KK$.
	Let $a_1, \dots, a_s$ be selected at random independently and uniformly from $S$.
	Then
	\[
		\Pr\Big[ \left\{(x-a_i)^{e_i} \;: 1 \leq i \leq s \right\}  \text{ is linearly independent} \Big] \geq 1 - \frac{s(s-1)}{|S|}
	\]
\end{corollary}
\begin{proof}
	We will prove this by induction on $s$: for $s = 1$, we always obtain a linearly independent family.
	We now consider $e = (e_1, \dots, e_{s+1}) \in \mP_{s+1}$, and we define the following events on possible outcomes $(a_1, \dots, a_{s+1})$:
	\[
	\begin{array}{rllr}
		A &= \{  \{(x-a_i)^{e_i} & \;: 1 \leq i \leq s+1 \}  &\text{ is linearly independent} \} \cr
		B &= \{  \{(x-a_i)^{e_i} & \;: 1 \leq i \leq s \}  &\text{ is linearly independent} \} \cr
		C &= \{  \{(x-a_i)^{e_i - 1} & \;: 1 \leq i \leq s \}  &\text{ is linearly independent} \}
	\end{array}
	\]
	Notice that $C \subseteq B \subseteq A$. Using Lemma \ref{lemmaprojection} and the induction hypothesis we obtain $\Pr(C) \geq 1 - \frac{s(s-1)}{|S|}$.
	%PK We have $\Pr(A) = \Pr(C) \cdot \Pr(A|C)$ and
        %PK using Corollary \ref{coro:finiteShiftedPowers},
From Proposition~\ref{prop:finiteShiftedPowers}
        we have $\Pr(A|C) \geq 1 - \frac{2s}{|S|}$.
	Since $\Pr(A) = \Pr(C) \cdot \Pr(A|C)$ we obtain the inequality
	\[
		\Pr(A) \geq \left(1 - \frac{s(s-1)}{|S|}\right) \cdot \left(1 - \frac{2s}{|S|}\right)  \geq 1 - \frac{s(s+1)}{|S|}.
	\]
\end{proof}

\bigskip

In the remainder of this section, our goal is to prove Theorem \ref{th:highproba} which states the following: if $a_1, \dots, a_s \in \KK$ are selected independently and uniformly at random  from a big enough finite set $S$, then with high probability  $(x-a_i)^{e_i} \, : \, 1 \leq i \leq s$ are linearly independent for all $e = (e_1,\ldots,e_s) \in \mP_s$.
A key ingredient of this proof is to have a bound on $\max(e_i)$ for a sequence $e \in \mP_s$ such that the elements of the family might be linearly dependent.
This bound is derived from the following structural result proved in \cite{GKP}.
  
\begin{proposition} \cite[Proposition 3.11]{GKP} \label{prop:linIndepAffPow}
	Suppose $\sum_{i=1}^s \alpha_i(x - a_i)^{e_i} = 0$ with $\alpha_i \not= 0$ and $(a_i,e_i) \not= (a_j,e_j)$ for all $i,j$.
	Then $\max(e_i) < \frac{s^2}{2} - 1$.
\end{proposition}

Motivated by this bound, we naturally define the bounded version of $\mP_s$: $\mP_s' = \{ e \in \mP_s : \max(e_i) \leq \frac{s^2}{2} - 2\}$.
The next Corollary ensures that if an outcome $(a_1, \dots, a_s)$ yields a linearly independent family for any $e \in \mP_s'$, then it also yields a linearly independent family for any $e \in \mP_s$.
%Its proof uses the following technical lemma:
%\begin{lemma}\label{lemmaprojection2}
%	Let $e = (e_1,\ldots,e_{s}) \in \mP_{s}$ for $s \geq 3$. then $f = (f_1,\ldots,f_s) \in \mP_s'$, where $f_i = \min\left(e_i, \frac{s^2}{2} - 2 \right)$.
%\end{lemma} 
%\begin{proof}
%	We have $n_i' = n_i$ for $i \leq \frac{s^2}{2} - 2$, and $n_i' = s \leq i$ for $i > \frac{s^2}{2} - 2$ since $s \geq 3$.
%\end{proof}

\begin{lemma}\label{lemmaprojection2}
	Let $e = (e_1,\ldots,e_{s}) \in \mP_{s}$ for $s \geq 2$. It we take $f_i$ such that $\min\{e_i, \frac{s^2}{2} - s\} \leq f_i \leq \frac{s^2}{2} - 2$ for all $i$, then $f = (f_1,\ldots,f_s) \in \mP_s'$.
\end{lemma} 
\begin{proof}
	We have $n_i' = n_i$ for $i \leq \frac{s^2}{2} - s + 1$, and $n_i' \leq s \leq i$ for $i \geq \frac{s^2}{2} - s  + 2$. % since $s \geq 3$.
\end{proof}

\begin{corollary}\label{aequalsb}
	We define the following events on possible outcomes $(a_1, \dots, a_s)$:
	\[
	\begin{array}{rllr}
		A &= \{  \bigwedge_{e \in \mP_{s}}   &\left\{(x-a_i)^{e_i} \;: 1 \leq i \leq s \right\}  &\text{ is linearly independent} \} \cr
		B &= \{  \bigwedge_{e \in \mP_{s}'}  &\left\{(x-a_i)^{e_i} \;: 1 \leq i \leq s \right\}  &\text{ is linearly independent} \}
	\end{array}
	\]
	Then $A = B$.
\end{corollary}
\begin{proof}
	 We first observe that if $a_1 = \cdots = a_s$, then $A = B$ trivially because  $\{(x-a_i)^{e_i} \;: 1 \leq i \leq s \}$ is linearly
	independent for every $e_1,\ldots,e_s$, so let us assume that they are not all equal.

	Since $\mP_s' \subseteq \mP_s$, we have $A \subseteq B$.
	Given an outcome $a = (a_1, \dots, a_s) \in B$ and a sequence $e \in \mP_s$, we can distinguish two cases.
	If $e \in \mP_s'$, then  the $s$ shifted powers $(x-a_1)^{e_1},\ldots,(x-a_s)^{e_s}$
        are linearly independent since $a \in B$.
	Otherwise, assume there exists $\alpha_i$ such that $\sum_{i = 1}^s \alpha_i (x-a_i)^{e_i} = 0$.
	We denote by $I = \{ i : e_i > s^2/2 - 2\}$, and, using Proposition \ref{prop:linIndepAffPow}, we have that $\alpha_i = 0$ for $i \in I$.
	We therefore rewrite the equality as $\sum_{i=1}^s \alpha_i (x - a_i)^{f_i} = 0$,  where $f_i  := e_i$ for $i \notin I$; otherwise, $f_i$ is chosen 
	in $\{s^2/2 - s,\ldots,s^2/2 - 2\}$ in such a way that there are no two equal $(a_i,f_i)$ (we observe that we can always choose $f_i$ in 
	this interval since there are $s-1$ possible values to choose
         from and there are at least two different $a_i$'s).
	Using Lemma \ref{lemmaprojection2}, we have that $f \in \mP_s'$ and thus $(x-a_i)^{f_i} \;: 1 \leq i \leq s$ are linearly independent since $a \in B$, proving that all the $\alpha_i$'s are zero.
\end{proof}

\medskip

Now that we have restricted our attention to a finite set $\mP_s'$,
we can directly use the union bound on all possible sequences $e \in \mP_s'$ to obtain the result using Corollary \ref{coro:highproba} for a fixed sequence.
We only need to obtain a upper bound on $|\mP_s'|$; this is done in following Proposition.
More precisely, we will compute exactly $|\mP_{s,d}|$ where $\mP_{s,d} = \{ e \in \mP_s : \max(e_i) < d\}$.
%PK dépacé plus bas:
%We will a similar proof technique as the one of  \cite{LorRie79} to count the number of P\'olya matrices in the context of Birkhoff interpolation.
\begin{proposition}\label{prop:countingPolya}
	Let $s \leq d$ be integers. Then
	\[
		|\mP_{s,d}| = \binom{s+d}{s}\frac{d+1-s}{d+1}
	\]
\end{proposition}
\begin{proof}
	Notice first that a sequence $e \in \mP_{s,d}$ can be represented by the $d$-tuple $(m_1, \dots m_{d})$, where $m_i = |\{ j : e_j = i-1 \}|$.
	Therefore, there is a bijection between $\mP_{s,d}$ and the following set:
	\[
		Q_{s,d} = \left\{ (m_1, \dots, m_{d}) \in \NN^d \;:\; \forall j \leq d,\:   \sum_{i=1}^j m_i \leq j  \,\bigwedge\, \sum_{i=1}^{d} m_i = s \right\}
	\]
	For each $m \in Q_{s,d}$, we associate a lattice path on the Cartesian plane as follows: start the path at $(0,0)$, and at the $i$-th step move right $1$ unit then go up $m_i$ units.
	The resulting path ends at position $(d,s)$, and never goes above the diagonal $y = x$.
	In fact, there is a bijection between $Q_{s,d}$ and the monotonic lattice paths starting at position $(0,0)$, ending at position $(d,s)$, and not passing above the diagonal $y = x$.
	These numbers are usually called \emph{ballot numbers}, and have been studied since de Moivre (1711).
	The analytic expression can be found in \cite[p. 451]{Knu406}, proving the result.
\end{proof}

In the Birkhoff interpolation paper  \cite{LorRie79} a similar proof technique
was used to count the number of P\'olya matrices.

\begin{theorem} \label{th:highproba}
	Let $S$ be a finite subset of $\KK$.
	Let $a_1, \dots, a_s$ be selected at random independently and uniformly from $S$.
	Then
	\[
		\mathrm{Pr}\left[ \bigwedge_{e \in \mP_s} \left\{(x-a_i)^{e_i} \;: 1 \leq i \leq s \right\}  \text{ is linearly independent} \right] \geq 1 - \frac{f(s)}{|S|}
	\]
	where $f(s) =  \binom{s + \frac{s^2}{2} - 1}{s} (s-1)(s-2)$.
\end{theorem}
\begin{proof}
	Following the notation of Corollary \ref{aequalsb}, we have $\Pr(A) = \Pr(B)$. We compute $\Pr(B)$ using the union bound:
	\[
		\Pr(B) \geq 1 - \sum_{e \in \mP_{s}'} \Pr\left( \left\{(x-a_i)^{e_i} \;: 1 \leq i \leq s \right\}  \text{ is linearly dependent} \right)
	\]
	By Corollary \ref{coro:highproba} we have $\Pr(B) \geq 1 - |\mP_s'| \cdot \frac{s(s-1)}{|S|}$.
	Using Proposition \ref{prop:countingPolya} with $d = \frac{s^2}{2} - 1$, we obtain
	\[
		\Pr(B) \geq  1 - \binom{s + \frac{s^2}{2} - 1}{s} \cdot \frac{\frac{s^2}{2} - s}{\frac{s^2}{2}} \cdot \frac{s(s-1)}{|S|}.
	\]
\end{proof}

\begin{remark}
	One can improve the previous lower bound by noticing that for some sequences $e \in \mP_s'$, we have $\Pr(\{(x-a_i) \} \text{ is linearly dependent}) = 0$.
	This is the case for instance for any sequence $e$ with distinct $e_i$'s, thus we have 
	\[
		\Pr(B) \geq 1 - \left(|\mP_s'| - \binom{s^2/2}{s}\right) \cdot \frac{s(s-1)}{|S|}.
	\]
\end{remark}

\section{Dimension lower bounds} \label{sec:dim}

As sufficient conditions for linear independence are hard to find, we now try to lower bound the dimension of the span of families of affine powers instead.
Of course, this can only be easier since linear independence implies full dimension of the corresponding space.
We mostly investigate two different cases: when all the exponents are big, as in Sections~\ref{sec:realind} and~\ref{sec:complexind}, and when small exponents are allowed.
In the latter case, we will require that the family $F$ satisfies the \emph{P\'olya condition} from the Introduction.
%PK : $\forall k, |\{ e_i \leq k \}| \leq k + 1$.
%This condition is a classical requirement in linear independence settings: if a family does not satisfy this condition, then it is linearly dependent.
Under this condition, we first prove an easy lower bound which holds
over any field.
\begin{proposition}
	Consider any family $F = \{ (x - a_i)^{e_i} \;:\: 1 \leq i \leq s \}$ satisfying the P\'olya condition.
	%\\
	Then we have $\dim \langle F \rangle \geq \sqrt{s}$.
\end{proposition}
\begin{proof}
	We partition $F$ according to the values of the exponents: $F = \cup_{i=1}^t F_i$ with $F_i = \{ (x-a_{i,j})^{d_i} \;:\: 1 \leq j \leq t_i \}$, and $d_i \not= d_j$ for $i \not= j$.
	The fact that $F$ satisfies the P\'olya condition implies that $d_i \geq t_i - 1$, and therefore that every $F_i$ is an independent family, using Proposition \ref{prop:WaringLinearIndependence}.
	\\
	Therefore, if there exists $k \in [|1;t|]$ such that $F_{k}$ contains at least $\sqrt{s}$ elements we have directly $\dim \langle F \rangle \geq \dim \langle F_k \rangle = |F_k| \geq \sqrt{s}$.
	Otherwise, we must have $t \geq \sqrt{s}$.
	 We consider a family $G$ % = \{ (x-a_{i,1})^{d_i} \;:\: 1 \leq i \leq t\}$ 
	obtained by taking one element in each $F_i$.
	Since the $d_i$'s are distinct,
        %PK this is a stepped family which is therefore independent,
        the elements of $G$ are linearly independent.
        This proves that $\dim \langle F \rangle \geq |G| = t \geq \sqrt{s}$.
\end{proof}

\medskip
In the following, we will show that we can achieve a linear lower bound for both real and complex field.

\subsection{The real case}
The following result is a consequence of Proposition \ref{prop:realLinearIndependence}
\begin{proposition}
	Consider any family $F = \{ (x - a_i)^{e_i} \;:\: 1 \leq i \leq s \}$, with $a_i \in \RR$ satisfying the P\'olya condition.
	\\
	Then we have $\dim \langle F \rangle \geq \lfloor \frac{s+4}{3} \rfloor$.
\end{proposition}
\begin{proof}
Assume that $e_1 \geq e_2 \geq \cdots \geq e_s$ and consider the family $G := \{ (x - a_i)^{e_i} \;:\: 1 \leq i \leq t \} \subset F$ with $t := \lfloor \frac{s+4}{3} \rfloor$. Since $F$ satisfies the P\'olya condition, we have 
that $e_{s-i} \geq i$ and, hence, $e_i \geq s-t$ for all $i \in \{1,\ldots,t\}$. The inequality $s-t \geq 2t-4$ holds, thus we conclude by Proposition \ref{prop:realLinearIndependence} that the elements of $G$ are linearly independent and
$\dim \langle F \rangle  \geq \dim \langle G \rangle = |G| = t$.
\end{proof}

\medskip
With more work, we can achieve a better lower bound.

Consider a family $F = \{(x-a_i)^{e_i}\, : 1 \leq i \leq s\}$. By a {\em sequence in the node $a \in \RR$} we
mean a maximal interval of integers $O$ such that for all $e \in O$, the element $(x-a)^e$ belongs to $F$.
A sequence $O$ with an odd size is naturally called an {\em odd sequence}.
The following result is just a restatement of  \cite[Corollary 9.(ii)]{GMK15}; this result was obtained by transforming
the problem of linear independence of shifted powers into an equivalent problem in Birkhoff interpolation, and then applying 
a celebrated result of Atkinson and Sharma~\cite{atkinson69} concerning real Birkhoff interpolation (see \cite[Theorem~1.5]{Lorentz84}).

\begin{corollary} \label{ordercor}
Consider a family $F = \{ (x - a_i)^{e_i} \;:\: 1 \leq i \leq s \}$, with $a_i \in \RR$ satisfying the P\'olya condition and set  $d := \max(e_i)$.
If every odd sequence $O$ in any node $a_i$  satisfies that $\max(O) = d$, then the elements of $F$ are linearly independent.
\end{corollary}

\begin{proposition}
	Consider any family $F = \{ (x - a_i)^{e_i} \;:\: 1 \leq i \leq s \}$, with $a_i \in \RR$ satisfying the P\'olya condition.
	\\
	Then we have $\dim \langle F \rangle \geq \lfloor \frac{s}{2} \rfloor + 1$.
\end{proposition}
\begin{proof}
	Let us denote $d = \max(e_i)$. If there is only one $i$ such that $d = e_i$ we have $\dim \langle F \rangle  = \dim \langle F \setminus \{(x-a_i)^{e_i}\} \rangle + 1$.
	So, without loss of generality we assume that there are at least two $e_i$'s equal to $d$. 
	Let $O_1,\ldots,O_k$ be all the odd sequences such that $d \notin O_i$; notice that $k \leq s-2$. 
	We denote $m_i = \min(O_i)$, $M_i = \max(O_i)$ and call $b_i$ the corresponding node; we order $O_1, \ldots, O_k$ so that $m_1 \leq \ldots \leq m_k$.
	We claim that $F' := F \setminus \{(x-b_i)^{m_i} \, : \, 1 \leq i \leq \lceil k/2 \rceil \}$ yields a linearly independent family.
	Since the size of $F'$ is $s - \lceil k/2 \rceil \geq
s - \lceil (s-2)/2 \rceil = \lfloor \frac{s}{2} \rfloor + 1$, this will prove that $\dim \langle F \rangle \geq \dim \langle F' \rangle \geq \lfloor \frac{s}{2} \rfloor + 1$.

	More precisely, we are going to prove that the  family
	$F' \cup \{(x-b_i)^{M_i+1} \, : \,\lceil k/2 \rceil + 1 \leq i \leq k\}$ is linearly independent.
 	Indeed, by construction of $F'$, every odd sequence $O$ in this family satisfies $\max(O) = d$; this is because we have removed an element from $O_1,\ldots,O_{\lceil k/2 \rceil}$ and added one to $O_{\lceil k/2 \rceil + 1},\ldots,O_k$, which converts all these odd sequences into even sequences. 
	Moreover, the new set also satisfies the P\'olya condition since for every $1 \leq i < j \leq k$, we have $m_i \leq m_j \leq M_j + 1$, so removing a shifted power of exponent $m_i$ and adding one of exponent $M_j + 1$ can never cause a violation of the P\'olya condition.
	By Corollary \ref{ordercor} we are done.
\end{proof}

It is worth pointing out that this result does not only bound the dimension of $\langle F \rangle$ but also shows how to explicitly obtain a linearly independent subset of $F$ of size
at least  $\lfloor \frac{s}{2} \rfloor + 1$. This will not be the case in the complex setting, where we will obtain a lower bound on the dimension of $\langle F \rangle$ but we will not provide an explicit linearly independent subset of $F$ of size $\Omega(s)$.

We do not know if the above bound is sharp. In the following example we exhibit a family $F$ of $s$ shifted powers that satisfy the P\'olya condition and we prove that
$\dim  \langle F \rangle = (3s-1)/4$.

\begin{lemma}\label{lowdim} Let $d \in \ZZ^+$ such that $d \, \equiv \, 2 \ ({\rm mod} \ 4)$ and consider $F_1 := \{ x^{i} \, : \, i$ odd and $i < d\}$ and $F_2 := \{ (x+1)^{i}, (x-1)^{i} \, : \,i$ even and $(d+2)/2 \leq i \leq d\}$.  The family $F := F_1 \cup F_2$ has $d+1$ elements, satisfies the P\'olya condition and $\dim \langle F \rangle = (3d + 2)/4.$ 
\end{lemma}
\begin{demo}It is easy to see that $|F_1| = d/2,\ |F_2| = (d+2)/2$ and that $F$ satisfies the P\'olya condition. We remark that $F$ has an element of degree $i$ for all $i \in \{1,3,\ldots,d/2,(d/2) +1,\ldots,d\}$. This implies  $\dim\langle F \rangle \geq (3d+2)/4$.
Moreover, for every $i$ even such that $(d+2)/2 \leq i \leq d$ we observe that $(x+1)^{i} - (x-1)^{i} = \sum_{j < i \atop j {\text \ odd}} 2\, \binom{i}{j}\, x^j \in \langle F_1 \rangle$. Hence, we can combine every pair elements of $F_2$ of the same degree so that the combination can be expressed as a linear combination of the elements of $F_1$. This proves that $\dim \langle F \rangle \leq |F| - \frac{|F_2|}{2} = d+1 - \frac{d+2}{4} = \frac{3d +2}{4}.$

\end{demo}

%be an even number and consider $F_1 := \{ x^{2i} \, : \, 0 \leq i < e\}$ and $F_2 := \{ (x+1)^{e + 2i + 1}, (x-1)^{e + 2i + 1}, \, : \, 0 \leq i < e/2\}$.  The family $F := F_1 \cup F_2$ has $2e$ elements, satisfies P\'olya condition and $\dim \langle F \rangle = 3e/2.$ 
%\end{lemma}
%\begin{demo}It is straightforward to check that $F$ satisfies P\'olya condition. We also remark that $F$ has an element of degree $i$ for all $i \in \{0,2,\ldots,e-2,e,e+1,\ldots,2e-1\}$, this makes $\dim\langle F \rangle \geq 3e/2$.
%Moreover, for every $0 \leq i < e/2$ we observe that $(x+1)^{e + 2i + 1} - (x-1)^{e + 2i + 1} = \sum_{j \leq e+2i \atop j {\text \ even}} 2\, \binom{e+2i-1}{j}\, x^j \in \langle F_1 \rangle$. Hence, we can combine every pair elements of $F_2$ of the same degree so that the combination can be expressed as a linear combination of the elements of $F_1$. This proves that $\dim \langle F \rangle \leq |F| - \frac{|F_2|}{2} = 2e - \frac{e}{2} = \frac{3e}{2}.$
%\end{demo}

\subsection{The complex case} \label{complexlb}
%The following result is a corollary of Proposition \ref{prop:nodesprop:RootSDEMultiplicity} in the case where $\forall i, e_i \geq s$.
%\begin{corollary}\label{bigExponentLowerBound}
%	Fix an integer $s \in \NN$ and consider any family $F = \{ (x - a_i)^{e_i} \;:\; 1 \leq i \leq s \}$ with $\forall i, e_i \geq s$.
%	Then we have $\dim \langle F \rangle > \frac{s}{2}$.
%\end{corollary}
%\begin{proof}
%	We extract a subfamily $G \subseteq F$ that is a basis of $\langle F \rangle$, and we will show that $|G| > \frac{s}{2}$.
%	Set $t = |G|$ and assume for contradiction that $t \leq \frac{s}{2}$.
%	Using Proposition \ref{prop:smallequation} with the choice of parameters $(2t-1, 2t-1)$, we get an $\SDE(k,  2t - 1)$ with $k \leq 2t - 1$ satisfied by all the elements of $G$:
%	\[
%		\sum_{i = 0}^{k} P_i(x) f^{(i)}(x) = 0
%	\]
%	
%	Since $G$ is a basis of $\langle F \rangle$, it follows that not only elements of $G$ are solutions, but also elements of $F$.
%	Since for any $i \in [|1;s|]$, we have $e_i \geq s \geq 2t-1 \geq k$, it follows from Proposition \ref{prop:nodesprop:RootSDEMultiplicity} that $s \leq \deg P_{k} \leq 2t-1 \leq s - 1$, a contradiction.
%\end{proof}
%\medskip
%Following the same proof layout, we now generalize previous result for even larger exponents.

In the complex case, our results rely on Corollary \ref{prop:nodesprop:RootSDEMultiplicity} and on a good choice of parameters in Proposition \ref{prop:smallequation}.

%\begin{proposition}\label{prop:bigExponentLowerBound}
%	Fix an integer $s \in \NN$ and consider any family $F = \{ (x - a_i)^{e_i} \;:\; 1 \leq i \leq s \}$ with $e_i \geq  s$.\\
%	Then we have $\dim \langle F \rangle > (2- \sqrt{2})s$.
%\end{proposition}
%\begin{proof}
%	We extract a basis $G \subseteq F$ of $\langle F \rangle$. We set $t := |G|$ and we assume for contradiction that $t \leq (2- \sqrt{2})s$.
%	We take $k = s$ and $l = s-1$, we observe that the inequality in Proposition \ref{prop:smallequation} is satisfied; indeed, the polynomial function
%	\[ 
%		p(T) = T^2 - (2l + 2k + 3)T + 2(k+1)(l+1) = T^2 - (4s+1)T + 2s(s+1)
%	\]
%	is positive for all $T \in [0,\gamma]$, where $\gamma := \frac{ (4s+1) - \sqrt{ (4s+1)^2 - 8s(s+1) }}{2} > (2 - \sqrt{2})s$.
%	Hence, by Propostion \ref{prop:smallequation}, there exists an $\SDE(t,k',s-1)$ with $t \leq k' \leq s$ satisfied by all the elements of $G$.
%	Since $G$ is a basis of $\langle F \rangle$, it follows that not only the elements of $G$ are solutions of this SDE, but also the elements of $F$.
%	Since for any $i \in [|1;s|]$, we have $e_i \geq s \geq k'$, it follows from Corollary \ref{prop:nodesprop:RootSDEMultiplicity} that $s \leq \deg P_{k'} \leq s - 1$, a 
%	contradiction.
%\end{proof}

\begin{proposition}\label{prop:bigExponentLowerBound2}
	Fix an integer $p \in \NN$ and  $\alpha > 0$  and consider any family $F = \{ (x - a_i)^{e_i} \;:\; 1 \leq i \leq p \}$ with $e_i \geq \alpha p$ for all $i$.
	\\
	Then we have $\dim \langle F \rangle > (1 + \alpha - \sqrt{\alpha^2 + 1})p$.
\end{proposition}
\begin{proof}
	We extract a basis $G \subseteq F$ of $\langle F \rangle$.
	We set $s := |G|$, $\gamma := (1 + \alpha - \sqrt{\alpha^2 + 1})p$ and we assume for contradiction that $s \leq \gamma$.
	We take $t = s$, $k = \alpha p$ and $l = p-1$, we claim that the inequality in Proposition \ref{prop:smallequation} is satisfied.
	Indeed, according to Proposition \ref{prop:smallequation}, it suffices to prove that the polynomial function 	
	\[ 
		\begin{array}{lll} 
		p(S) & = & S^2 - (2l + 2k + 3)S + 2(k+1)(l+1)   \\
	 		 & = & S^2 - (2p + 2 \alpha p + 1)S + 2p(\alpha p+1) 
		\end{array}
	\]
	is positive for $S = s$. We will prove that $p(S) > 0$ for all $S \in [0,\gamma]$. For this purpose, we consider 
	\[ 
		\begin{array}{lllll} 
		q(S) & :=  & p(S) + S - 2p & = & S^2 - (2p + 2 \alpha p)S + 2 \alpha p^2 ;
		\end{array}
	\] 
	it is straightforward to check that $q(S) \geq 0$ for all $S \in [0, \gamma]$ since $\gamma$ is the smallest root of $q$. 
	We conclude that $p(S) > 0$ in $[0, \gamma]$ because the following inequalities hold  for every $S \in [0,\gamma]$:
    \[ 
    	p(S) = q(S) - S + 2p \geq q(S) - \gamma + 2p = q(S) + (\sqrt{\alpha^2+1} - \alpha + 1)p > q(S).
    \]

	Hence, by Proposition \ref{prop:smallequation}, there exists an $\SDE(s,k',p-1)$ with $s \leq k' \leq \alpha p$ satisfied by all the elements of $G$.
	Since $G$ is a basis of $\langle F \rangle$, it follows that not only the elements of $G$ are solutions of this SDE, but also the elements of $F$.
	Since for any $i \in [|1;p|]$ we have $e_i \geq \alpha p \geq k'$, it follows from Corollary \ref{prop:nodesprop:RootSDEMultiplicity} that $p \leq \deg P_{k'} \leq p - 1$, a contradiction.
\end{proof}

\smallskip
\begin{corollary}\label{prop:bigExponentLowerBound}
	Fix an integer $s \in \NN$ and consider any family $F = \{ (x - a_i)^{e_i} \;:\; 1 \leq i \leq s \}$ with $e_i \geq  s$ for all $i$.\\
	Then we have $\dim \langle F \rangle > (2- \sqrt{2})s$.
\end{corollary}

\medskip
We now extend Proposition~\ref{prop:bigExponentLowerBound2} %PK previous result
  by allowing small exponents.
  %PK , according that the family satisfies the P\'olya condition:
\begin{theorem}
	For any family $F = \{ (x - a_i)^{e_i} \;:\; 1 \leq i \leq s \}$ satisfying the P\'olya condition, we have $\dim \langle F \rangle \geq (1 - \frac{\sqrt{2}}{2}) (s-1) $.
\end{theorem}
\begin{proof}
	We start by dropping the $\lceil \frac{s}{2} \rceil$ terms with least exponents and call $F'$ the resulting family.
	By construction we have $|F'| = s - \lceil \frac{s}{2} \rceil =\lfloor \frac{s}{2} \rfloor$, and because $F$ satisfies the P\'olya condition, every element $(x-a)^e \in F'$ must verify $e \geq \lceil \frac{s}{2} \rceil$.
	Thus the family $F'$ satisfies the hypothesis of Corollary \ref{prop:bigExponentLowerBound}, implying that $\dim \langle F' \rangle \geq (2 - \sqrt{2}) \lfloor \frac{s}{2} \rfloor$.
	The result follows since $\langle F' \rangle$
        is a subset of $\langle F \rangle$.
\end{proof}

\bigskip

\section{Conclusion}
We have proposed two conjectures concerning linear independence of shifted powers;  we have
proved their counterparts for the field of real numbers and have given some steps towards potential proofs over $\CC$. The proof of Conjecture \ref{conj:GMK},
apart from being interesting by itself, has nice consequences in arithmetic complexity: it implies a linear lower bound on the number of shifted powers
needed to represent a polynomial of degree $d$. We believe that
both of them are true; however, we have the feeling that a tool different from shifted differential equations should be used to prove them.

We have also provided bounds on the dimension of the vector space spanned by a family $F$ of shifted powers that satisfy the P\'olya condition.  The lower bounds for the field of complex numbers
  given in  Section~\ref{complexlb} are nonconstructive in the sense that
  they do not pinpoint a linearly independent subset of $F$ of cardinality
  equal to our lower bound on $\dim F$ (but they of course imply the existence
  of such a subset). It would be interesting to obtain a constructive proof.
  This may be related to the problem of obtaining a ``good'' sufficient condition for the linear independence of shifted powers over $\CC$.

To our knowledge, the family $F$ of real or complex polynomials that satisfies the P\'olya condition and that spans a vector space with the least dimension is the one we provide in Lemma \ref{lowdim}. It would be interesting to improve the bounds we provide or to show that they are tight.

\section*{Acknowledgements}

The authors are supported by ANR project
CompA (code ANR--13--BS02--0001--01). I. Garc\'ia-Marco was also partially supported by MTM2016-78881-P. We thank the anonymous referees for suggesting several improvements in the presentation of this paper.

%\bibliography{biblio}{}
%\bibliography{../../../Biblio/biblio}{}
\bibliographystyle{plain}

\end{document}